\documentclass[reqno,11pt]{amsart}

\usepackage{latexsym,geometry} 
\usepackage{amsmath,amsthm,amsopn,amstext,amscd,amsfonts,amssymb}

\usepackage{eqnarray,amsrefs,color}

\usepackage{latexsym,enumerate,graphicx}

\usepackage{fullpage,yhmath}
\usepackage{hyperref}
\usepackage{eucal}
\usepackage{color}
\usepackage{verbatim}
\usepackage[normalem]{ulem}

\usepackage[colorinlistoftodos]{todonotes}

\newtheorem{theorem}{Theorem}[section]
\newtheorem{lemma}[theorem]{Lemma}
\newtheorem{prop}[theorem]{Proposition}

\newtheorem{defi}[theorem]{Definition\rm}
\newtheorem{rem}[theorem]{Remark\/}

\newcommand{\R}{\mathbb{R}}

\numberwithin{equation}{section}
\newcommand{\dsum}{\displaystyle\sum}

\begin{document}

\title[Singular anisotropic elliptic equations]{Boundedness of solutions to singular anisotropic elliptic equations }

\author{Barbara Brandolini}
\address[Barbara Brandolini]{Dipartimento di Matematica e Informatica, Universit\`a degli Studi di Palermo, via Archirafi 34, 90123 Palermo, Italy}
\email{barbara.brandolini@unipa.it}

\author{Florica C. C\^irstea}
\address[Florica C. C\^irstea]{School of Mathematics and Statistics, The University of Sydney,
	NSW 2006, Australia}
\email{florica.cirstea@sydney.edu.au}

\date{}
\keywords{Anisotropic operator; boundary singularity;  bounded solutions}

\subjclass[2010]{ 35J75, 35J60, 35Q35}

\begin{abstract}
We  prove 
the uniform boundedness of all solutions for a general class of Dirichlet anisotropic elliptic problems 
of the form $$-\Delta_{\overrightarrow{p}}u+\Phi_0(u,\nabla u)=\Psi(u,\nabla u) +f $$ on a bounded open subset $\Omega\subset \mathbb R^N$ $(N\geq 2)$, where $ \Delta_{\overrightarrow{p}}u=\sum_{j=1}^N \partial_j (|\partial_j u|^{p_j-2}\partial_j u)$ and  
$\Phi_0(u,\nabla u)=\left(\mathfrak{a}_0+\sum_{j=1}^N \mathfrak{a}_j |\partial_j u|^{p_j}\right)|u|^{m-2}u$,  
with $\mathfrak{a}_0>0$,
$m,p_j>1$,   $\mathfrak{a}_j\geq 0$ for $1\leq j\leq N$ and $N/p=\sum_{k=1}^N (1/p_k)>1$. We assume that $f \in L^r(\Omega)$ with $r>N/p$. The feature of this study  is the inclusion of a possibly singular gradient-dependent term $\Psi(u,\nabla u)=\sum_{j=1}^N |u|^{\theta_j-2}u\, |\partial_j u|^{q_j}$, where $\theta_j>0$ and $0\leq q_j<p_j$ for $1\leq j\leq N$. The existence of such weak solutions is contained in a recent paper by the authors.
\end{abstract}

\maketitle



\section{Introduction}\label{Sec1}

Singular elliptic and parabolic equations are widely recognised as 
one of the important branches of modern analysis from the viewpoint of the physical significance of the problems at hand and the
novel analytical methods that they generate (see, for example, \cite{Du,V1,V2} and the references therein).
 Singularities can develop or manifest in many ways for elliptic PDEs, such as the blow-up of a solution at an isolated point (see, for instance, \cite{CD}) or at the whole boundary (see, for example, \cite{CD1,ZD}). Singularities can also arise through the coefficients in an equation, usually in the form of singular potentials (see, e.g., \cite{WD}). Here, we tackle singularities that appear in anisotropic elliptic PDEs via the nonlinear lower order term going to infinity when a positive solution tends to 0 on the boundary. 

The current demand to predict with a higher and higher degree of accuracy the behaviour in systems arising from the physical and biological sciences means
that more variables need to be taken into account in the modelling of the natural phenomena, generating 
nonlinear PDEs with nonstandard growth and featuring singularities \cite{AS}. Anisotropic elliptic and parabolic  equations represent a very active research area in the last decades. The increasing interest in nonlinear anisotropic problems is justified by their  
applications in many areas from image recovery and the mathematical modelling of non-Newtonian fluids to biology, where they serve as models for the propagation of epidemic diseases in heterogeneous domains (see, for example, \cite{ADS}).  
  
Important tools available for the isotropic case cannot be extended to the anisotropic setting (such as the strong maximum principle, see \cite{V}).
Despite
important strides on existence, uniqueness, and regularity of weak solutions of anisotropic equations (see, for instance, \cites{ACCZ,AC,CV,dBFZ,DiCastro,FVV,MT}), the influence of anisotropy in many problems remains as yet elusive in many essential aspects. This note is a continuation of our study initiated in \cite{baci,BC} to prove existence and boundedness of solutions for general singular anisotropic elliptic equations.   

Throughout this paper, we assume that $\Omega$ is a bounded, open subset of $\mathbb R^N$ ($N\geq 2$). We impose no smoothness assumption on the boundary of $\Omega$. 
We always assume that 
\begin{equation} \label{IntroEq0}
p_j\in (1,\infty)\ \mbox{for every }1\leq j\leq N \quad  \text{and}\quad 
p<N,
\end{equation}
where $p:=N/\sum_{j=1}^N (1/p_j)$ is the harmonic mean of $p_1,\ldots, p_N$.
For any $r>1$, let $r'=r/(r-1)$ be the conjugate exponent of $r$. 
We set $\overrightarrow{p}=\left(p_1,p_2,\ldots,p_N\right)$ and 
$\overrightarrow{p}'=(p_1',p_2',\ldots,p_N')$. We define by $W_0^{1,\overrightarrow{p}}(\Omega)$       
the closure of   
$C_c^\infty (\Omega)$, the set of smooth functions with compact supports in $\Omega$,  
with respect to the norm 
$$ \|u\|_{W_0^{1,\overrightarrow{p}}(\Omega)}=\sum_{j=1}^N \|\partial_j u\|_{L^{p_j}(\Omega)}.$$ As usual, $\nabla u=(\partial_1u,\ldots,\partial_N u)$ is the gradient of $u$.  
The assumption $p<N$ gives that the embedding $W_0^{1,\overrightarrow{p}}(\Omega)\hookrightarrow L^s(\Omega)$ is continuous 
for every
$s\in [1,p^\ast]$ and compact for every $s\in [1,p^\ast)$, 
where $p^\ast:=Np/(N-p)$ stands for the anisotropic Sobolev exponent.   
The anisotropic $\overrightarrow{p}$-Laplacian operator given by  
\begin{equation} \label{jdmsn} 
-\Delta_{\overrightarrow{p}} u=-\sum_{j=1}^N  \partial_j (|\partial_j u|^{p_j-2}\partial_j u)
\end{equation} 
is a bounded, coercive and pseudo-monotone operator of Leray--Lions type from 
$W_0^{1,\overrightarrow{p}}(\Omega)$ into its dual  $W^{-1,\overrightarrow{p}'}(\Omega)$. 
For every $(t,\xi)\in \mathbb R\times \mathbb R^N$, we define 
\begin{equation} \label{phiz} 
 \Phi_0(t,\xi)=\left(\mathfrak a_0+ \dsum_{j=1}^N \mathfrak{a}_j |\xi_j|^{p_j}\right)|t|^{m-2} t, 
 \end{equation}
where  
$m>1$, $ \mathfrak a_0>0$, $ \mathfrak a_j\geq  0$ for $ 1\leq j\leq N$, whereas
\begin{equation} \label{ass}
 \Psi(t,\xi)=\dsum_{j=1}^N  |t|^{\theta_j-2}t\, |\xi_j|^{q_j} \end{equation}
with $\theta_j>0$ and 
$0\leq q_j<p_j$ for all $1\leq j\leq N$.   

Let $f\in L^1(\Omega)$. In our recent paper \cite{BC}, we give sufficient conditions for the existence of solutions for general singular anisotropic elliptic equations, 
including for our model problem
\begin{equation} \label{eq1}
\left\{ \begin{aligned}
& -\Delta_{\overrightarrow{p}} u+ \Phi_0(u,\nabla u)=\Psi(u,\nabla u)+f \quad \mbox{in } \Omega,\\ 
&  u\in W_0^{1,\overrightarrow{p}}(\Omega), \quad \Phi_0(u,\nabla u) \in L^1(\Omega).
\end{aligned} \right.
\end{equation}

\noindent We distinguish the following cases according to the values of the $\theta_j$'s.

\vspace{0.2cm}
{\bf Case 1:} ({\em Non-singular}) For every $1\leq j\leq N$, we have $\theta_j>1$. 

\vspace{0.2cm}
{\bf Case 2:} ({\em Mildly singular}) There exists $1\leq j\leq N$ such that $\theta_j\leq 1$. 
In this situation, we will consider non-negative solutions of \eqref{eq1} and assume, in addition, that $f \ge 0$ a.e. in $\Omega$.

\vspace{0.2cm}
As in \cite{BC}, by a solution of \eqref{eq1}, we understand the following.

\begin{defi} \label{defi1} 	
A function $u\in W_0^{1,\overrightarrow{p}}(\Omega)$, which is non-negative in Case~2, is said to be a solution of \eqref{eq1} if 
$\Phi_0(u,\nabla u) \in L^1(\Omega)$ and for every $v\in W_0^{1,\overrightarrow{p}}(\Omega)\cap L^\infty(\Omega)$, we have   
\begin{equation} \label{bbb1}
\displaystyle \displaystyle I_u(v):=\int_{\{|u|>0\}} \Psi(u,\nabla u)\,v\,dx\quad \mbox{is finite} 
\end{equation}
and
\begin{equation}\label{ff1}
\sum_{j=1}^N\int_\Omega |\partial_j u|^{p_j-2}\partial_j u\, \partial_j v\, dx+ 
  \int_\Omega \Phi_0(u,\nabla u) \,v\,dx = I_u(v) +\int_\Omega f\, v\, dx.
\end{equation} 
\end{defi}

Under suitable conditions on $m$, we proved in \cite{BC} that \eqref{eq1} has a solution in the above sense. Before recalling the relevant result, we 
introduce the following sets: 
\begin{equation} \label{np}
\begin{aligned}
& N_{\overrightarrow{\mathfrak{a}}}:=\left\{1\leq j\leq N:\  
\mathfrak{a}_j q_j=0,\ \  \frac{\theta_j p_j}{p_j-q_j} \geq p
\right\}, \\
& P_{\overrightarrow{\mathfrak{a}}}:=\left\{1\leq j\leq N :\ 
\mathfrak{a}_j q_j>0,\ 
\mathfrak{m}_j>1\right\},\quad \mbox{where } \mathfrak{m}_j:= \frac{p_j-q_j}{q_j} \left( \frac{ \theta_j p_j}{p_j-q_j} - p \right).
\end{aligned}
\end{equation}

 Suppose that $m$ satisfies
  \begin{equation}\label{m}
	m> \max_{j\in N_{\overrightarrow{\mathfrak{a}}}}  \frac{\theta_j p_j}{p_j-q_j}\quad  
	\text{and} \quad  m>\min\left\{\theta_j , \mathfrak{m}_j\right\} \ \  \text{for every } j\in P_{\overrightarrow{\mathfrak{a}}}.
	\end{equation} 

As a consequence of the main results in \cite{BC}, we obtain the following existence result. 

 \begin{theorem} \label{mainth}  Let \eqref{IntroEq0}--\eqref{ass} and \eqref{m} hold. 
 Assume Case 1 or Case 2.
 If $f\in L^{(p^\ast)'}(\Omega)$, then \eqref{eq1} has a solution $u\in W_0^{1,\overrightarrow{p}}(\Omega) $, which satisfies that 
 \begin{equation} \label{satis}
 \Phi_0(u,\nabla u)\,u,\ \Psi(u,\nabla u)\,u\in L^1(\Omega)\quad \mbox{and}\quad \eqref{ff1}  \ \mbox{holds for } v=u.
 \end{equation}
  \end{theorem}
We point out that, if $f\in L^1(\Omega)$ and 
 $\min_{1\leq j\leq N} \mathfrak{a}_j>0$, then in \cite{BC} we also prove that \eqref{eq1} has a solution $u\in W_0^{1,\overrightarrow{p}}(\Omega) $. 
 
\subsection{Main result} 
In this note, in the framework of Theorem \ref{mainth}, we prove that, if $f\in L^r(\Omega)$ with $r>N/p$, 
then every solution of \eqref{eq1} satisfying \eqref{satis} is bounded. Moreover, every bounded solution satisfies a uniform bound. More precisely, we prove the following.

\begin{theorem} \label{pre} 
Let \eqref{IntroEq0}--\eqref{ass} and \eqref{m} hold. 
 Assume Case 1 or Case 2. If $f\in L^{r}(\Omega)$ with $r>N/p$,
 then for every solution 
 $u\in W_0^{1,\overrightarrow{p}}(\Omega) $ of \eqref{eq1} satisfying \eqref{satis}, we have $u\in L^\infty(\Omega)$.  Moreover, there exists a positive constant $C$, depending only on $\Omega, N, m,\overrightarrow{p}, \overrightarrow{q},r, \|f\|_{L^r(\Omega)}, \{ \mathfrak{a}_j\}_{0\leq j\leq N}$ and $\{\theta_j\}_{1\leq j\leq N}$, such that
\begin{equation}\label{uinfty}  
 \|u\|_{L^\infty(\Omega)}\le C.
\end{equation}
\end{theorem}

We remark that our assumption \eqref{m} is equivalent to requiring that,
for every $1 \le j\le N$, we have one of the following: 
\begin{equation} \label{cen4}
\begin{aligned}
& \mathfrak{a}_j>0, \> q_j \ge 0, \> 0<\theta_j<\max\{m,p\}\frac{p_j-q_j}{p_j}+\frac{m q_j}{p_j}, 
\\
& \mathfrak{a}_j=0, \> q_j \ge 0,\> 0<\theta_j<\max\{m,p\}\frac{p_j-q_j}{p_j}.
\end{aligned}
\end{equation}

\subsection{Strategy of the proof}

The crucial ingredient in the proof of our Theorem \ref{pre} is the celebrated Stampacchia's method (see Lemma 4.1 in \cite{S}, which is recalled in Section 2).  
Even if we know that Stampacchia's procedure is nowadays a typical technique to prove that a solution of an elliptic equation is bounded, its application in this context is far from being trivial. For the reader's convenience, we provide  all the details, and we split the proof into several propositions of increasing generality until we reach the full range in \eqref{cen4}.  

\subsection{Further comments} 

We conclude our introduction by comparing our results with the ones known in the isotropic case, for which the model problem is
\begin{equation}\label{iso}
\left\{ \begin{array}{ll} 
-\Delta_{p} u+ \lambda |u|^{m-2}u=g(u) |\nabla u|^q +f \quad& \mbox{in } \Omega,\\  \\
u=0 & \mbox{on } \partial\Omega.
\end{array} \right.
\end{equation}
Assume that $\lambda=0$. The exhaustive analysis of the necessary growth condition on $g$ to obtain a solution for every datum $f\in L^r(\Omega)$, where $r>1$, was made by Porretta and Segura de Le\'on in \cite{PS}. When $f \in L^{\frac N p}(\Omega)$ and $\lim_{|s| \to +\infty}g(s)=0$, they prove the existence of a bounded weak solution $u \in W_0^{1,p}(\Omega) \cap L^\infty (\Omega)$. Under the same assumptions, in \cite{LMS} the authors investigate the case $p=2$, $1<q<2$ and prove an analogous existence result of a bounded solution of \eqref{iso}.

When $g(u) $ is singular in the $u$-variable at $u=0$, that is $g(u)=|u|^{-\nu}$ with $\nu>0$, and we take $q=p$, the existence of a bounded weak solution is proved in \cite{GM,GPS} when $f\in L^r(\Omega)$ with $r>N/ p$. 
 In \cite{GM} the presence of an absorption term, with exponent $m=2$, is used to prove the existence of a bounded solution in $H^1_{\rm{loc}}(\Omega)$ when $p=q=2$ and $f$ is a bounded, non-negative function.  In \cite{GPS1} the authors prove existence results when $p=q=2$, $\nu\in (0,1)$ and $f\in L^r(\Omega)$ with $r\geq N/2$,  allowing the presence of a sign-changing datum $f$.  They also discuss related questions such as the existence of solutions when the datum $f $ is less
regular, or the boundedness of the solutions when $r>N/2$.


\section{ Notation and auxiliary results}

For $k>0$, we let $T_k:\mathbb R\to \mathbb R$ stand for the truncation at height $k$, that is,  
\begin{equation} \label{trunc} 
T_k(s)=s \quad \mbox{if } |s| \le k, \quad T_k(s)=k\, \frac{s}{|s|} \quad \mbox{if }  |s|>k.
\end{equation}
\noindent Moreover, we define $G_k:\mathbb R\to \mathbb R$ by  
\begin{equation}\label{gk}
G_k(s)=s-T_k(s)\quad \mbox{for every }s\in \R,
\end{equation}
 so that $G_k=0$ on $[-k,k]$.  
 
 In order to prove the boundedness of a solution to our problem, we will make use of the following result, contained in the celebrated paper \cite{S} by G. Stampacchia.
 
 \begin{lemma}[See Lemma~4.1 in \cite{S}] \label{st}
 Let $k_0\in \R$ and $\varphi: [k_0,+\infty) \to [0,+\infty)$ be a non-increasing function. Assume that, for some positive constants $a$, $b$ and $C$, we have
 $$
 \varphi(h)\le \frac{C}{(h-k)^a}[\varphi(k)]^b\quad \mbox{for all } h>k\ge k_0.
 $$
Then:
 \begin{enumerate}[$i)$]
 \item if $b>1$, it holds
 $$
 \varphi(k_0+d)=0
 $$
 where
 $$
 d=C^{\frac 1 a}[\varphi(k_0)]^{\frac{b-1}{a}}2^{\frac{b}{b-1}};
 $$
 
 \item if $b=1$, it holds
 $$
 \varphi(h)\le e\, \exp\left(-\zeta(h-k_0)\right)\varphi(k_0)
 $$
 where
 $$\zeta=\left(eC\right)^{-\frac 1 a};
 $$
 
 \item if $b<1$ and $k_0>0$, it holds
 $$
 \varphi(h) \le 2^{\frac{s}{1-b}}\left(C^{\frac{1}{1-b}}+(2k_0)^{s} \varphi(k_0)\right)h^{-s}
 $$ 
 where
 $$
 s=\frac{a}{1-b}.
 $$
 \end{enumerate}
 \end{lemma}
 
 We next recall the anisotropic Sobolev inequality in \cite{T}*{Theorem~1.2}.

\begin{lemma} \label{sob} Let $N\geq 2$ be an integer. If $1 < p_j <\infty$ for every $1\le j \le N$ and $p<N$,  
then there exists a
positive constant $\mathcal{S}=\mathcal{S}(N,\overrightarrow{p})$, such that 
\begin{equation} \label{send} 
\|u\|_{L^{p^\ast}(\R^N)}\leq \mathcal{S}  \prod_{j=1}^N \| \partial_j  u\|_{L^{p_j}(\R^N)}^{1/N} \quad \mbox{for all } u\in C_c^\infty(\R^N),
\end{equation}
where, as usual, $p^\ast:=Np/(N-p)$.
\end{lemma}

\begin{rem} \label{an-sob} Let $\Omega$ be a bounded, open subset of $\R^N$ with $N\geq 2$. If  $1 < p_j <\infty$ for every $1\le j \le N$ and $p<N$, then by a density argument, 
\eqref{send} extends to all $u\in  W_0^{1,\overrightarrow{p}}(\Omega) $ so that the arithmetic-geometric mean inequality yields
\begin{equation}\label{ASI}
\|u\|_{L^{p^\ast}(\Omega)}\leq \mathcal{S}  \prod_{j=1}^N \| \partial_j  u\|_{L^{p_j}(\Omega)}^{1/N} \leq 
\frac{\mathcal{S}}{N} \sum_{j=1}^N \| \partial_j  u\|_{L^{p_j}(\Omega)}=\frac{\mathcal{S}}{N}\|u\|_{W_0^{1,\overrightarrow{p}}(\Omega)}
\end{equation}
for all $u\in W_0^{1,\overrightarrow{p}}(\Omega)$. Moreover, 
using H\"older's inequality, the embedding $W_0^{1,\overrightarrow{p}}(\Omega)\hookrightarrow L^s(\Omega)$ is continuous 
for every
$s\in [1,p^\ast]$ and compact for every $s\in [1,p^\ast)$.
\end{rem}
 
In the sequel, we need the following version of Young's inequality. 

\begin{lemma}[Young's inequality] \label{young} Let $N\geq 2$ be an integer. Assume that $\beta_1,\ldots, \beta_N$ are positive numbers
	and $1<R_k<\infty$ for each $1\leq k\leq N-1$. If $\sum_{k=1}^{N-1} (1/R_k)<1$, then for every $\delta>0$, there exists
	a positive constant $C_\delta$ (depending on $\delta$) such that 
	$$ \prod_{k=1}^N \beta_k\leq \delta\sum_{k=1}^{N-1} \beta_k^{R_k} +C_\delta \,\beta_N^{R_N},  
	$$ where we define $R_N=\left[1-\sum_{k=1}^{N-1} (1/R_k)\right]^{-1}$. 
\end{lemma}


\section{Boundedness of solutions}\label{s3}

We first assume that
\begin{equation}\label{cen}
\theta_j<m \quad \mbox{and} \quad \quad q_j \ge 0,  \quad \mbox{ for every } 1 \le j \le N.
\end{equation}

\begin{prop}\label{prop1}
Let \eqref{IntroEq0}--\eqref{ass} hold and $f\in L^{r}(\Omega)$ with $r>N/p$. 
 Assume Case~1 or Case~2.  
If $\min_{1 \le j \le N}\mathfrak{a}_j>0$ and \eqref{cen} holds, 
then for every solution 
 $u\in W_0^{1,\overrightarrow{p}}(\Omega) $ of \eqref{eq1} satisfying \eqref{satis}, we have $u\in L^\infty(\Omega)$. Moreover, there exists a positive constant $C$, which depends only on $\Omega, N, m,\overrightarrow{p}, \overrightarrow{q},r, \|f\|_{L^r(\Omega)}, \{ \mathfrak{a}_j\}_{0\leq j\leq N}$ and $\{\theta_j\}_{1\leq j\leq N}$, such that
\begin{equation}\label{uinfty1}  
 \|u\|_{L^\infty(\Omega)}\le C.
\end{equation}
\end{prop}

\begin{proof} 
 Let $u\in W_0^{1,\overrightarrow{p}}(\Omega) $ be a solution of  \eqref{eq1} satisfying \eqref{satis}. 
In what follows, we take $k\geq k_0>1$ as large as needed and define 
\begin{equation}\label{mu}
 A(k):=\{x\in \Omega: \ |u(x)|\geq k\},\quad \mu_u(k):=\mbox{meas}\,(A(k)).
\end{equation}
Note that $G_k(u)=u-T_k(u)\in W_0^{1,\overrightarrow{p}}(\Omega)$ and $\partial_j G_k(u)=\partial_j u$ on $A(k)$ for all $1\leq j\leq N$.  
The property \eqref{satis} enables us to take $v=G_k(u)$ as a test function in \eqref{ff1}. Hence, we have 
\begin{equation} \label{fifa}
\begin{aligned}
& \sum_{j=1}^N\int_{A(k)}|\partial_j u|^{p_j}\, dx+ \mathfrak a_0 \int_{A(k)} |u|^{m-1} |G_k(u)|\,dx+
\sum_{j=1}^N \mathfrak a_j  \int_{A(k)} |u|^{m-1} |\partial_j u|^{p_j}|G_k(u)|\,dx\\
&  \leq \sum_{j=1}^N \int_{A(k)} |u|^{\theta_j-1} |\partial_j u|^{q_j} |G_k(u)|\,dx   +\int_{A(k)} |f|\, |G_k(u)|\, dx.
\end{aligned}
\end{equation} 
In what follows, we let $\varepsilon\in (0,1)$ be small such that
\begin{equation} \label{eps}
3N \varepsilon<\mathfrak{a}_0, \quad N \varepsilon < \min_{j=1,\ldots,N}\mathfrak{a}_j
\end{equation}
and we let $k> k_0$ with $k_0>1$ large satisfying 
\begin{equation}\label{k0}
k_0\ge  (1/\varepsilon)^{1/(m-\max_{1\leq i\leq N} \theta_i)}.
\end{equation}
Let $1\leq j\leq N$ be fixed. 

(a) Assume first that $q_j=0$. By our hypotheses, we have $\theta_j<m$ so that
\begin{equation} \label{ford}
 \begin{aligned} 
\int_{A(k)} |u|^{\theta_j-1}  |G_k(u)|\,dx &\leq k^{\theta_j-m} \int_{A(k)} |u|^{m-1} \,|G_k(u)|\,dx\\
& \leq \varepsilon \int_{A(k)} |u|^{m-1} |G_k(u)|\,dx.\end{aligned}
\end{equation}

(b) Assume that $q_j>0$. We define \begin{equation} \label{etaj} 
\eta_j:= \theta_j-1- \frac{\left(m-1\right) q_j}{p_j}.\end{equation}

We distinguish the following sub-cases:

\vspace{0.2cm}
(b$_1$) Let $\eta_j\leq 0$. For simplicity of notation, we define
\begin{equation} \label{rak} R_k(u):=   \int_{A(k)} |G_k(u)|\,dx.\end{equation}
By H\"older's inequality and Young's inequality, for every $\varepsilon\in (0,1)$, there exists a constant $C_\varepsilon>0$, depending on $\varepsilon$, but independent of $k$, such that
\begin{equation} \label{n1}
 \begin{aligned} 
& \int_{A(k)} |u|^{\theta_j-1} |\partial_j u|^{q_j} |G_k(u)|\,dx\\
& \leq k^{\eta_j} \left( \int_{A(k)} |u|^{m-1} |\partial_j u|^{p_j}|G_k(u)|\,dx\right)^\frac{q_j}{p_j}
\left[ R_k(u)\right]^\frac{p_j-q_j}{p_j}\\
&\leq \varepsilon  \int_{A(k)} |u|^{m-1} |\partial_j u|^{p_j}|G_k(u)|\,dx+C_\varepsilon \,R_k(u).
\end{aligned}
\end{equation}

(b$_2$) Let $\eta_j>0$, where $\eta_j$ is defined in \eqref{etaj}. Remark that 
$$ 1-\frac{q_j}{p_j}-\frac{\eta_j}{m-1}=\frac{m-\theta_j}{m-1}\in (0,1).  
$$
Thus, using H\"older's inequality and \eqref{rak}, we arrive at
\begin{equation} \label{art2} 
\begin{aligned} 
&\int_{A(k)} |u|^{\theta_j-1} |\partial_j u|^{q_j} |G_k(u)|\,dx\\
& \leq 
 \left( \int_{A(k)} |u|^{m-1} |\partial_j u|^{p_j}|G_k(u)|\,dx\right)^\frac{q_j}{p_j} \left(\int_{A(k)} |u|^{m-1} |G_k(u)|\,dx \right)^\frac{\eta_j}{m-1}
  [R_k(u)]^{\frac{m-\theta_j}{m-1}}.
\end{aligned} \end{equation}
By Young's inequality in Lemma~\ref{young}, for every $\varepsilon\in (0,1)$, there exists a constant $C_\varepsilon>0$ such that
the right-hand side of the inequality in \eqref{art2} is bounded above by 
\begin{equation}\label{n2}
 \varepsilon  \int_{A(k)} |u|^{m-1} |\partial_j u|^{p_j}|G_k(u)|\,dx+\varepsilon \int_{A(k)} |u|^{m-1} |G_k(u)| \,dx+C_\varepsilon \,R_k(u).
\end{equation}
Letting $k_0$ sufficiently large so that
\begin{equation}\label{k00}
k_0\ge\left(\frac{3NC_\varepsilon}{\mathfrak{a}_0}\right)^{\frac{1}{m-1}},
\end{equation}
from the inequalities \eqref{fifa}--\eqref{n2}, we  obtain
\begin{equation} \label{mel} 
\sum_{j=1}^N\int_{A(k)}|\partial_j u|^{p_j}\, dx +\frac{\mathfrak{a}_0}{3}\int_{A(k)}|u|^{m-1}|G_k(u)|\, dx\leq  \int_{A(k)} |f|\, |G_k(u)|\, dx.
 \end{equation} 
 Now we focus our attention on the left-hand side of \eqref{mel}. By Lemma~\ref{sob}, there exists a positive constant
$\mathcal{S}=\mathcal{S}(N,\overrightarrow{p})$ such that 
\begin{equation} \label{art1}  
\begin{aligned}
\left(\int_{A(k)}|G_k(w)|^{p^\ast}\, dx \right)^{\frac{1}{p^\ast}}& \leq \mathcal S 
\prod_{j=1}^N \left(\int_{A(k)}|\partial_j G_k(w)|^{p_j}\, dx\right)^{\frac{1}{Np_j}}\\
& \leq 
\mathcal S 
 \left( \sum_{\ell=1}^N\int_{A(k)}|\partial_\ell w|^{p_\ell}\, dx\right)^\frac{1}{p}
\end{aligned}
\end{equation} for all $w\in W_0^{1,\overrightarrow{p}}(\Omega)$. Moreover, taking $h>k$ yields
\begin{equation}\label{nnn}
(h-k)^{p^\ast}\mu_w(h)\le \int_{A(k)}|G_k(w)|^{p^\ast}\,dx.
\end{equation}

We now distinguish different cases.

($\bullet$) Let $m>p$.
Thus, by using the H\"older inequality, we can bound from above the right-hand side of \eqref{mel} as follows
\begin{equation}\label{z7}
\begin{aligned}
\int_{A(k)} |f|\, |G_k(u)|\,dx 
&\le \int_{A(k)}|f|\, dx+ \int_{A(k)} |f|\,|G_k(u)|^p\, dx
\\
& \le\int_{A(k)}|f|\, dx+ \|f\|_{L^r(\Omega)} \left(\int_{A(k)}|G_k(u)|^{pr'}\,dx\right)^{\frac{1}{r'}}.
\end{aligned}
\end{equation} 
Since, by assumption, $r> N/p$, which implies $p r'<p^\ast$, by the interpolation inequality we get that there exists $0<\xi<1$ such that $1/(pr')=\xi/p^\ast+(1-\xi)/p$ and 
$$
\left(\int_{A(k)}  |G_k(u)|^{pr'}\, dx \right)^{\frac{1}{r'}} 
\le \left(\int_{A(k)}  |G_k(u)|^{p^\ast}\, dx \right)^{\frac{\xi p}{p^\ast}}\left(\int_{A(k)}  |G_k(u)|^p\, dx \right)^{1-\xi}.
$$
Fix $\delta >0$. Young's inequality yields the existence of a positive constant $C_\delta$ such that
\begin{equation}\label{z8}
\begin{aligned}
& \|f\|_{L^r(\Omega)} \left(\int_{A(k)}|G_k(u)|^{pr'}\right)^{\frac{1}{r'}} 
\\
&\le \delta \|f\|_{L^r(\Omega)}^{\frac 1 \xi} \left(\int_{A(k)}|G_k(u)|^{p^\ast}\, dx \right)^{\frac{p}{p^\ast}}+C_\delta\int_{A(k)} |G_k(u)|^p\, dx.
\end{aligned}
\end{equation}
From \eqref{mel}--\eqref{z8}, we deduce that
\begin{equation}
\begin{aligned}
&\frac{1}{\mathcal{S}^p}\left(\int_{A(k)}  |G_k(u)|^{p^\ast}\, dx \right)^{\frac{p}{p^\ast}}+\frac{\mathfrak{a}_0}{3}\int_{A(k)}|u|^{m-1}\, |G_k(u)|\, dx
\\
 &\le \delta \|f\|_{L^r(\Omega)}^{\frac 1 \xi} \left(\int_{A(k)} |G_k(u)|^{p^\ast}\, dx \right)^{\frac{p}{p^\ast}}+
C_\delta \int_{A(k)}  |G_k(u)|^p \, dx
+ \int_{A(k)}|f|\, dx.
\end{aligned}
\end{equation}
 We choose $\delta>0$ small enough  and $k_0>1$ large enough so that
 $$
 \delta \le \frac{\|f\|_{L^r(\Omega)}^{-\frac 1 \xi}}{2\mathcal{S}^p}, \quad k_0> \left(\frac{3\mathcal{S}^pC_\delta}{\mathfrak{a}_0}\right)^{\frac{1}{m-p}}.
 $$ 
Finally,  by H\"older's inequality we have
\begin{equation}\label{mmm}
\left(\int_{A(k)}  |G_k(u)|^{p^\ast}\, dx \right)^{\frac{p}{p^\ast}} \le 2 \mathcal{S}^p \int_{A(k)}|f|\, dx
\le 2 \mathcal{S}^p\|f\|_{L^r(\Omega)}\mu_u(k)^{1-\frac 1 r}.
\end{equation}
By combining \eqref{nnn} with \eqref{mmm}, we arrive at
$$
\mu_u(h) \le \frac{C\> }{(h-k)^{p^\ast}}\mu_u(k)^{\frac{p^\ast}{pr'}},
$$
with 
$$C=\left(2\mathcal{S}^p\|f\|_{L^r(\Omega)}\right)^{\frac{p^\ast}{p}}.$$
Since $p^\ast>pr'$, we can apply Lemma \ref{st} $i)$. Hence,  $\mu_u(k_0+d)=0$ with $d$ given by
$$
d=C^{\frac {1}{ p^\ast}}\mu_u(k_0)^{\frac{1}{pr'}-\frac{1}{p^\ast}}2^{\frac{N(r-1)}{pr-N}}\le C^{\frac {1}{ p^\ast}}|\Omega|^{\frac{1}{pr'}-\frac{1}{p^\ast}}2^{\frac{N(r-1)}{pr-N}}.
$$
This immediately implies that $u \in L^\infty (\Omega)$ with
$$
\|u\|_{L^\infty(\Omega)}\le k_0+d.
$$
 
\smallskip
($\bullet\bullet$) We now let $m \le p$.
Since $f\in L^r(\Omega)$ with $r>N/p$ and $p<N$, we have $1/\gamma+1/r<1$ for every $\gamma>p^\ast/p$. 
Assume that $u\in L^\gamma(\Omega)$ for some $\gamma\geq  p^\ast$.  
Hence, by H\"older's inequality, we find that
\begin{equation} \label{ford0} 
\int_{A(k)} |f|\, |u|\, dx\leq \|u\|_{L^\gamma(\Omega)} \|f\|_{L^r(\Omega)} \left(\mu_u(k)\right)^{1-\frac{1}{\gamma}-\frac{1}{r}}.
\end{equation}

For any $\Gamma> p^\ast/p$, we define 
\begin{equation} \label{ite} \beta(\Gamma):=\left(1-\frac{1}{\Gamma}-\frac{1}{r}\right)\frac{p^\ast}{p}>0.\end{equation}

Using \eqref{mel}, \eqref{art1}, \eqref{nnn} and \eqref{ford0}, we obtain that
\begin{equation} \label{neo}
\mu_u(h)\leq \frac{\mathcal{C}\>\|u\|_{L^\gamma(\Omega)}^{\frac{p^\ast}{p}}}{ \left(h-k\right)^{p^\ast}}  
\left(\mu_u(k)\right)^{
\beta(\gamma) },
\end{equation}
where
$$
\mathcal{C}=\mathcal{S}^{p^\ast} \|f\|_{L^r(\Omega)}^{\frac{p^\ast}{p}}.
$$
We claim that 
\begin{equation}\label{b1}
\mbox{there exists }\gamma=\gamma(N,p,r) \ge p^\ast \mbox{  such that }
\beta(\gamma)>1.
\end{equation} 
If so, from Lemma \ref{st} $i)$ we deduce that
 $\mu_u(k_0+d)=0$ with $d$ given by
$$
d=\mathcal{C}^{\frac {1}{ p^\ast}} \|u\|_{L^\gamma(\Omega)}^{\frac 1 p} \mu_u(k_0)^{\frac{\beta(\gamma)-1}{p^\ast}}2^{\frac{\beta(\gamma)}{\beta(\gamma)-1}}.$$
By Young's inequality, for every $\tau >0$ there exists $C_\tau>0$ such that  
$$
d\le \mathcal{C}^{\frac {1}{ p^\ast}}\left(\tau \|u\|_{L^\infty(\Omega)}+C_\tau\right) |\Omega|^{\frac{1}{pr'}-\frac{1}{ p^\ast}}2^{\frac{\beta(\gamma)}{\beta(\gamma)-1}}.
$$
Since $\|u\|_{L^\infty(\Omega)}\le k_0+d$, by choosing $\tau=\tau\left(N,p,r,|\Omega|,\|f\|_{L^r(\Omega)}\right)$ small, we conclude \eqref{uinfty1}.
 
In order to prove \eqref{b1}, we now distinguish two cases:

\vspace{0.2cm}
{\bf Case (A).} Let $\beta(\gamma)>1$ for $\gamma=p^\ast$. The claim follows. 

\vspace{0.2cm}
{\bf Case (B).} Assume that Case (A) fails. If $\beta(p^\ast)=1$, then we choose $\gamma_1\in (p^\ast/p,p^\ast)$ so that $\beta(\gamma_1)\in (0,1)$. Since $u\in L^{p^\ast}(\Omega)\subseteq L^{\gamma_1}(\Omega)$, we can take $\gamma=\gamma_1$ in \eqref{neo}.
Then, by Lemma~\ref{st} $iii)$, we find that $u\in L^{\gamma_2}(\Omega)$ for every $p^\ast<\gamma_2<p^\ast/(1-\beta(\gamma_1))$. 
 Hence, $\beta(\gamma_2)>\beta(p^\ast)=1$ and the claim follows with $\gamma=\gamma_2$. It remains to treat the case $\beta(p^\ast)<1$. 

\noindent Take $\gamma=\gamma_1=p^\ast$ in \eqref{neo}.
Then, by Lemma~\ref{st} $iii)$, we find that $u\in L^{\gamma_2}(\Omega)$ for every $p^\ast<\gamma_2<p^\ast/(1-\beta(\gamma_1))$. We distinguish three cases:

(B$_1$) If $\beta(\gamma_2)>1$, then the claim follows with $\gamma=\gamma_2$.

(B$_2$) If $\beta(\gamma_2)=1$, then $u \in L^{{\widetilde\gamma}_2}(\Omega)$ for any $\gamma_2<\widetilde \gamma_2<p^\ast/(1-\beta(\gamma_1))$. Since $\beta(\Gamma)$ is increasing in $\Gamma$, we have $\beta(\widetilde \gamma_2)>\beta(\gamma_2)=1$. As in the previous case, we conclude the claim with $\gamma=\widetilde \gamma_2$.

(B$_3$) If $\beta(\gamma_2)<1$, then $u \in L^{\gamma_3}(\Omega)$ for every $\gamma_3$ satisfying
$$
\frac{p^\ast}{1-\beta(\gamma_1)}<\gamma_3<\frac{p^\ast}{1-\beta(\gamma_2)}.
$$
We now compute $\beta(\gamma_3)$ and we iterate the previous cases (B$_1$)--(B$_3$). We claim that there exists $n \ge 3$ such that $u \in L^{\gamma_n}(\Omega)$ with $\beta(\gamma_n) \ge 1$ and $0<\beta(\gamma_i)<1$ for every $2\le i \le n-1$, where
\begin{equation}\label{pap11}
\frac{p^\ast}{1-\beta(\gamma_{i-1})}<\gamma_{i+1}<\frac{p^\ast}{1-\beta(\gamma_i)}.
\end{equation}
Indeed, if this were not the case, then $\{\beta(\gamma_i)\}_{i\geq 1}$ would be an increasing sequence converging to $L\leq 1$ as $i\to \infty$. Moreover, from \eqref{pap11}, we have
$$
\left(
1-\frac{1-\beta(\gamma_{i-1})}{p^\ast} -\frac 1 r\right)\frac{p^\ast}{p}  \leq \beta(\gamma_{i+1}) \le \left(
1-\frac{1-\beta(\gamma_{i})}{p^\ast} -\frac 1 r\right) \frac{p^\ast}{p}
$$ 
for every $i\geq 2$. 
By letting $i\to \infty$, we arrive at
$$ L=\left(1-\frac{1-L}{p^\ast} -\frac{1}{r} \right)\frac{p^\ast}{p}. 
$$ 
However, $r>N/p$ implies that $L>1$, which is a contradiction. 

If $\beta(\gamma_n)> 1$ (respectively, $\beta(\gamma_n)=1$),  with the same arguments as in (B$_1$) (respectively, (B$_2$)) with $\gamma_n$ instead of $\gamma_2$, we conclude our claim with $\gamma=\gamma_n$ (respectively, $\gamma=\widetilde \gamma_n$).
\end{proof} 

\begin{rem}\label{rem1}
Let us stress that in case (a), when $q_j=0$,  the restriction $\mathfrak{a}_j>0$ is not needed.
\end{rem}

\medskip

In our next result, instead of \eqref{cen}, we assume a potentially weaker condition, that is
\begin{equation} \label{cen2}
\begin{aligned}
&  \theta_j<m\ \mbox{whenever } 1\leq j\leq N\ \mbox{and } q_j>0,\\
& \theta_j<\max\{p,m\}\ \mbox{whenever } 1\leq j\leq N \ \mbox{and }q_j=0.
\end{aligned}
\end{equation}

\begin{prop} \label{prop2}
Let \eqref{IntroEq0}--\eqref{ass} hold and $f\in L^{r}(\Omega)$ with $r>N/p$. 
 Assume Case~1 or Case~2.  If $\min_{1 \le j \le N}\mathfrak{a}_j>0$ and 
\eqref{cen2} holds, then for every solution 
 $u\in W_0^{1,\overrightarrow{p}}(\Omega) $ of \eqref{eq1} satisfying \eqref{satis}, we have $u\in L^\infty(\Omega)$. Moreover, there exists a positive constant $C$, which depends only on $\Omega, N, m,\overrightarrow{p}, \overrightarrow{q},r, \|f\|_{L^r(\Omega)}, \{ \mathfrak{a}_j\}_{0\leq j\leq N}$ and $\{\theta_j\}_{1\leq j\leq N}$, such that
\begin{equation}\label{uinfty2}  
 \|u\|_{L^\infty(\Omega)}\le C.
\end{equation}
\end{prop}

\begin{proof} If $m\geq p$ or \eqref{cen} holds, then the claim follows from Proposition~\ref{prop1}. Hence, we assume that $m<p$ and there exists 
$1\leq j_0\leq N$ such that $ q_{j_0}=0$ and $m\leq \theta_{j_0}<p$.  
As in the proof of Proposition \ref{prop1}, 
we let $k>1$ and take $v=G_k(u)$ in \eqref{ff1} to recover \eqref{fifa}. 

Let $m < p$. Assume that $u \in L^\gamma(\Omega)$ for some $\gamma \geq p^\ast$. 
For $1 \le j\le N$, whenever $q_j>0$, we let $\eta_j$ be as in \eqref{etaj}. Let us define
$$
\begin{aligned}
&I_1:=\{1\le j \le N: q_j=0, \> m \le \theta_j <p\}, & &I_2:=\{1\le j \le N: q_j=0, \> \theta_j <m\},&
\\
&I_3:=\{1\le j \le N: q_j>0, \> \eta_j \le 0\}, & &I_4:=\{1\le j \le N: q_j>0, \> \eta_j >0, \> \theta_j <m\}.&
\end{aligned}
$$
Fix $\varepsilon>0$ small as needed and $k>k_0$ with $k_0>1$ sufficiently large. 

$(a_1)$ Assume that $I_1 \ne \emptyset$ and let $j \in I_1$. 
By H\"older's inequality and \eqref{art1} with $w=u$, we infer that
$$
\begin{aligned}
\int_{A(k)}|u|^{\theta_j-1}\, |G_k(u)|\, dx &\le \left(\int_{A(k)}|G_k(u)|^{p^\ast}\, dx \right)^{\frac{1}{p^\ast}}\left(\int_{A(k)}|u|^\gamma\, dx \right)^{\frac{\theta_j-1}{\gamma}}\left(\mu_u(k)\right)^{1-\frac{\theta_j-1}{\gamma}-\frac{1}{p^\ast}}
\\
&\le \mathcal{S}\left(\sum_{\ell=1}^N \int_{A(k)}|\partial_\ell G_k(u)|^{p_\ell}\, dx\right)^{\frac 1 p}\left(\int_{A(k)}|u|^\gamma\, dx \right)^{\frac{\theta_j-1}{\gamma}}\left(\mu_u(k)\right)^{1-\frac{\theta_j-1}{\gamma}-\frac{1}{p^\ast}}.
\end{aligned}
$$
By Young's inequality, there exists a positive constant $C_\varepsilon$ such that
\begin{equation}\label{car1}
\begin{aligned}
\int_{A(k)}|u|^{\theta_j-1}\, |G_k(u)|\, dx
\le &\varepsilon \sum_{\ell=1}^N \int_{A(k)}|\partial_\ell G_k(u)|^{p_\ell}\, dx
\\
&+C_\varepsilon \|u\|_{L^\gamma(\Omega)}^{p'(\theta_j-1)}\left(\mu_u(k)\right)^{p'\left(1-\frac{\theta_j-1}{\gamma}-\frac{1}{p^\ast}\right)}.
\end{aligned}
\end{equation}

$(a_2)$ Assume that $I_2 \cup I_3 \cup I_4 \ne \emptyset$ and let $j \in I_2 \cup I_3 \cup I_4$.  Then, arguing as in the proof of Proposition \ref{prop1}, using \eqref{ford} and \eqref{n1}--\eqref{n2} we arrive at
\begin{equation}\label{car2}
\begin{aligned}
\sum_{j \in I_2 \cup I_3\cup I_4} \int_{A(k)} |u|^{\theta_j-1}\, |\partial_j u|^{q_j}\, |G_k(u)|\, dx \le & \varepsilon \sum_{j \in I_3\cup I_4} \int_{A(k)} |u|^{m-1}\, |\partial_j u|^{p_j}\, |G_k(u)|\, dx
\\
&+\varepsilon \left(|I_2|+|I_4|\right)\int_{A(k)} |u|^{m-1}\, |G_k(u)|\, dx
\\
&+C_\varepsilon\left(|I_3|+|I_4|\right)R_k(u),
\end{aligned}
\end{equation}
where $R_k(u)$ is defined by \eqref{rak}.

Let $\theta=\max_{j \in I_1}\theta_j$. From \eqref{fifa}, \eqref{ford0}, \eqref{car1} and \eqref{car2}, choosing $\varepsilon\in (0,1/(2N))$ small as in \eqref{eps} and $k_0>1$ large as in \eqref{k0} and \eqref{k00}, we get that
$$
\begin{aligned}
\sum_{\ell=1}^N\int_{A(k)}|\partial_\ell u|^{p_\ell}\, dx \le &2C_\varepsilon \left(\sum_{j \in I_1} \|u\|_{L^\gamma(\Omega)}^{p'(\theta_j-1)}\left(\mu_u(k)\right)^{p'\left(1-\frac{\theta_j-1}{\gamma}-\frac{1}{p^\ast}\right)}\right)
\\
&+2\|u\|_{L^\gamma(\Omega)} \|f\|_{L^r(\Omega)} \left(\mu_u(k)\right)^{1-\frac{1}{\gamma}-\frac{1}{r}}
\\
\
\le& M \,\mu_u(k)^{\min\left\{p'\left(1-\frac{\theta-1}{\gamma}-\frac{1}{p^\ast}\right),1-\frac{1}{\gamma}-\frac{1}{r}\right\}},
\end{aligned}
$$
where  $C_\Omega$ is a positive constant depending on $|\Omega|$, $p$, $N$, and $r$, and
$$
M:=C_\Omega \,\left(C_\epsilon\left(\sum_{j \in I_1} \|u\|_{L^\gamma(\Omega)}^{p'(\theta_j-1)}\right)+\|u\|_{L^\gamma(\Omega)} \|f\|_{L^r(\Omega)} \right).
$$
Taking $h>k\ge k_0$ and using \eqref{art1} and \eqref{nnn}, we obtain that
\begin{equation} \label{neop}
\mu_u(h) \le \frac{\mathcal{S}^{p^\ast} M^{\frac{p^\ast}{p}}}{(h-k)^{p^\ast}}  \,\mu_u(k)^{\min \left\{\widehat\beta(\gamma), \beta(\gamma) \right\}},
\end{equation} 
where $\widehat\beta(\gamma)$ and $\beta(\gamma)$ are defined by 
\begin{equation} \label{pap1} 
 \widehat \beta(\gamma):= \left(1-\frac{\theta-1}{\gamma}-\frac{1}{p^\ast}\right)\frac{p^\ast}{p-1}, \quad \beta(\gamma):=  \left(1-\frac{1}{\gamma}-\frac{1}{r}
\right)\frac{p^\ast}{p}.
\end{equation} 
It is easy to check that 
\begin{equation} \label{dift} \widehat \beta(\gamma)-\beta(\gamma)= \left(\frac{1}{N(p-1)}+\frac{1}{pr}+\frac{2p-1-p\,\theta}{\gamma p(p-1)}
\right)p^\ast. 
\end{equation}
We distinguish two cases:

\vspace{0.2cm}
$(E_1)$ Let $2p-1-p\,\theta\geq 0$. Then, $\widehat\beta(\gamma)-\beta(\gamma)>0$. Hence, $\min\{\beta(\gamma),\widehat \beta(\gamma)\}=\beta(\gamma)$. In this case, we follow the argument given after \eqref{neo} in the proof
of Proposition~\ref{prop1}.

\vspace{0.2cm}
$(E_2)$ Let $2p-1-p\,\theta<0$. Then, from \eqref{dift}, we see that 
\begin{equation} \label{inc} 
\widehat \beta(\gamma)-\beta(\gamma)\ \mbox{ increases in } \gamma. \end{equation} 

\vspace{0.2cm}
$\bullet$ If $\widehat \beta(p^\ast)-\beta(p^\ast)>0$, then using that $\widehat \beta(\gamma)-\beta(\gamma)>0$ for every $\gamma\geq p^\ast$, we can proceed as in Case $(E_1)$ to finish the proof. 

\vspace{0.2cm}
$\bullet$ If $\widehat \beta(p^\ast)-\beta(p^\ast)\leq 0$, then $\min\{\beta(\gamma),\widehat \beta(\gamma)\}=\widehat \beta(\gamma)$ for every $\gamma\leq p^\ast$ and close to $p^\ast$.  As in the proof of Proposition~\ref{prop1} (after \eqref{neo}), working here with $\widehat \beta(\gamma)$ instead of $\beta(\gamma)$, we have two situations:  

{\bf Case (A).} Let $\widehat \beta(\gamma)>1$ for $\gamma=p^\ast$.  Then, by Lemma~\ref{st} $i)$ we get that 
$u\in L^\infty(\Omega)$.

{\bf Case (B).}  Let $\widehat \beta(p^\ast)\leq 1$. If $\widehat \beta(p^\ast)=1$, then 
we choose $\gamma_1<p^\ast$ and close to $p^\ast$ so that  
$\widehat \beta(\gamma_1)\in (0,1)$. 
Since $u\in L^{p^\ast}(\Omega)\subseteq L^{\gamma_1}(\Omega)$, we can take $\gamma=\gamma_1$ in \eqref{neop}.
Then, by Lemma~\ref{st} $iii)$, we find that $u\in L^{\gamma_2}(\Omega)$ for every $p^\ast<\gamma_2<p^\ast/(1-\widehat\beta(\gamma_1))$. 
We have $\widehat \beta(\gamma_2)>\widehat\beta(p^\ast)=1$ and 
$\beta(\gamma_2)>\beta(p^\ast)\geq \widehat \beta(p^\ast)=1$. Hence, $\min\{ \beta(\gamma_2),\widehat \beta(\gamma_2)\}>1$ and thus $u\in L^\infty(\Omega)$ as in Case~(A) with $\gamma=\gamma_2$. 

It remains to treat the case $\widehat\beta(p^\ast)<1$. Here, we take $\gamma_1=p^\ast$ and, 
as before, we find that $u\in L^{\gamma_2}(\Omega)$ for every $p^\ast<\gamma_2<p^\ast/(1-\widehat\beta(\gamma_1))$.
We distinguish three cases:

(B$_1$) 
If $\min\{\widehat\beta(\gamma_2),\beta(\gamma_2)\}>1$, then  
 $u\in L^\infty(\Omega)$ as in Case (A) with $\gamma=\gamma_2$. 
 
 (B$_2$) If $\min\{\widehat\beta(\gamma_2), \beta(\gamma_2)\}=1$, then by choosing $\gamma_2<\widetilde \gamma_2<p^\ast/(1-\widehat \beta(\gamma_1))$, since both $\widehat \beta(\Gamma)$ and $\beta(\Gamma)$ are increasing in $\Gamma$, we obtain that
$
 \min\{\widehat\beta(\widetilde\gamma_2), \beta(\widetilde\gamma_2)\}>1.
$
Hence, $u \in L^\infty (\Omega)$ as in case (A) with $\gamma=\widetilde \gamma_2$.

 (B$_3$)  If $\min\{\widehat\beta(\gamma_2), \beta(\gamma_2)\}<1$, then by
 Lemma~\ref{st} $iii)$, $u\in L^{\gamma_3}(\Omega)$ for every $\gamma_3$ satisfying 
$$\frac{p^\ast}{1-\min\{\widehat \beta(\gamma_1),\beta(\gamma_1)\}}<\gamma_3<\frac{p^\ast}{1- \min\{\widehat \beta(\gamma_2),\beta(\gamma_2)\}}.$$ 
 We now compute $\min\{ \widehat \beta(\gamma_3), \beta(\gamma_3)\}$ and we iterate the previous cases (B$_1$)--(B$_3$). We claim that there exists $n \ge 3$ such that $u \in L^{\gamma_n}(\Omega)$ with 
$\min\{\widehat\beta(\gamma_n),\beta(\gamma_n)\}\geq 1$ and $\min\{\widehat\beta(\gamma_i),\beta(\gamma_i)\}< 1$ for every $2\le i \le n-1$, where 
\begin{equation}\label{ppa}
\frac{p^\ast}{1-\min\{\widehat \beta(\gamma_{i-1}),\beta(\gamma_{i-1})\}}<\gamma_{i+1}<\frac{p^\ast}{1- \min\{\widehat \beta(\gamma_i),\beta(\gamma_i)\}}.
\end{equation}
Indeed, assume by contradiction that $\min\{\widehat\beta(\gamma_j),\beta(\gamma_j)\}< 1$ for all $j\geq 1$. 
We distinguish two cases: If there exists $i\geq 2$ such that $\widehat \beta(\gamma_i)> \beta(\gamma_i)$, then from \eqref{inc}, we have
$\widehat \beta(\gamma_j)> \beta(\gamma_j)$ for every $j\geq i$ and 
we reach a contradiction as in the  proof of Proposition~\ref{prop1} (after \eqref{pap11}). 
If, in turn, $\widehat \beta(\gamma_i)\le\beta(\gamma_i)$ for every $i\geq 1$, then $\widehat \beta(\gamma_i)$ is an increasing sequence converging to a limit $L \in (0,1]$ as $i \to \infty$. Moreover, from \eqref{ppa} we have
$$ \left(
1-\frac{\theta}{p^\ast} +\frac{(\theta-1) \widehat\beta(\gamma_{i-1})}{p^\ast} \right) \frac{p^\ast}{p-1}  < \widehat\beta(\gamma_{i+1})< 
\left(
1-\frac{\theta}{p^\ast} +\frac{(\theta-1) \widehat\beta(\gamma_{i})}{p^\ast} \right)\frac{p^\ast}{p-1} 
$$ 
for every $i\geq 2$. By letting $i\to \infty$, we arrive at
$$ L=\left(1-\frac{\theta}{p^\ast} +\frac{(\theta-1) L}{p^\ast} \right)\frac{p^\ast}{p-1}, 
$$ 
which gives that $L=(p^\ast-\theta)/(p-\theta)>1$ and this is an immediate contradiction with $L\le 1$. 

Finally, using that  $\min\{\widehat\beta(\gamma_n),\beta(\gamma_n)\}\geq 1$, we conclude that $u\in L^\infty(\Omega)$ with the same arguments as in (B$_1$), (B$_2$), respectively, with $\gamma_n$ instead of $\gamma_2$ and $\min\{\widehat \beta(\gamma_{n-1}),\beta(\gamma_{n-1})\}$ instead of $\widehat \beta(\gamma_1)$.
\end{proof}

\begin{rem}\label{rem2}
Let us stress that in case (a$_1$), when $j \in I_1$,  the restriction $\mathfrak{a}_j>0$ is not needed.
\end{rem}

\medskip

We next assume that 
\begin{equation} \label{cen3}
\theta_j<\max\{m,p\}\frac{p_j-q_j}{p_j}+\frac{m q_j}{p_j}, \quad q_j \ge 0, \> \mbox{ for every } 1\le j \le N.
\end{equation}

\begin{prop} \label{prop3}
Let \eqref{IntroEq0}--\eqref{ass} hold and $f\in L^{r}(\Omega)$ with $r>N/p$. 
 Assume Case~1 or Case~2.  If 
$\min_{1\leq j\leq N} \mathfrak{a}_j>0$ and \eqref{cen3} holds, then for every solution 
 $u\in W_0^{1,\overrightarrow{p}}(\Omega) $ of \eqref{eq1} satisfying \eqref{satis}, we have $u\in L^\infty(\Omega)$. Moreover, there exists a positive constant $C$, which depends only on $\Omega, N, m,\overrightarrow{p}, \overrightarrow{q},r, \|f\|_{L^r(\Omega)}, \{ \mathfrak{a}_j\}_{0\leq j\leq N}$ and $\{\theta_j\}_{1\leq j\leq N}$, such that
\begin{equation}\label{uinfty3}  
 \|u\|_{L^\infty(\Omega)}\le C.
\end{equation}
\end{prop}

\begin{proof} If $m\ge p$ or \eqref{cen} holds, then by Proposition~\ref{prop1}, we have $u\in L^\infty(\Omega)$. On the other hand, if $m<p$ and 
\eqref{cen2} holds, then by Proposition~\ref{prop2}, we get the claim. So, it remains to assume that  $m< p$ and
there exists $1\leq j_1\leq N$ such that 
$$q_{j_1}>0\quad \mbox{and}\quad m\leq \theta_{j_1}<p \left(1-\frac{q_{j_1}}{p_{j_1}}\right) + \frac{mq_{j_1}}{p_{j_1}}.$$  
As in the proof of Proposition \ref{prop1},  
we let $k>1$ and take $v=G_k(u)$ in \eqref{ff1} to recover \eqref{fifa}.

Assume that $u \in L^\gamma(\Omega)$ for some $\gamma \geq p^\ast$. We define $I_1,I_2,I_3$ and $I_4$ as in the proof of Proposition~\ref{prop2}. 
By assumption, we have $I_5\not=\emptyset$, where $I_5$ is given by 
$$ I_5:=\left\{1\leq j\leq N: \ q_j>0,\ m\leq \theta_j<p \left(1-\frac{q_{j}}{p_{j}}\right) + \frac{mq_{j}}{p_{j}}\right\}. 
$$

Fix $\varepsilon>0$ small as needed and $k>k_0$ with $k_0>1$ sufficiently large. 

If $j\in I_1$ or $j\in I_2\cup I_3\cup I_4$, we regain \eqref{car1} or 
\eqref{car2}, respectively. 

If $j\in I_5$, then by H\"older's inequality, we obtain that
\begin{equation} \label{mio} \begin{aligned}
& \int_{A(k)} |u|^{\theta_j-1}\, |\partial_j u|^{q_j}\, |G_k(u)|\, dx \\
& \le    \left(\int_{A(k)}|G_k(u)|^{p^\ast}\, dx \right)^{\frac{p_j-q_j}{p_j p^\ast}}
\left( \int_{A(k)} |u|^{m-1} |\partial_j u|^{p_j} |G_k(u)|\,dx\right)^\frac{q_j}{p_j}
 \\
&\quad \times \left(\int_{A(k)}|u|^\gamma\, dx \right)^{\frac{\eta_j}{\gamma}} \left(\mu_u(k)\right)^{
\frac{p_j-q_j}{p_j}\left(1-\frac{1}{p^\ast}\right)- \frac{\eta_j}{\gamma}},
\end{aligned} \end{equation}
where $\eta_j$ is given by \eqref{etaj}. 

Using \eqref{art1} in the right-hand side of \eqref{mio}, then Young's inequality, we arrive at
\begin{equation} \label{euu} \begin{aligned} 
\int_{A(k)} |u|^{\theta_j-1}\, |\partial_j u|^{q_j}\, |G_k(u)|\, dx \leq & \varepsilon  
 \sum_{\ell=1}^N\int_{A(k)}|\partial_\ell G_k(u)|^{p_\ell}\, dx\\
 &+\varepsilon 
  \int_{A(k)} |u|^{m-1} |\partial_j u|^{p_j} |G_k(u)|\,dx\\
  & + C_\varepsilon \| u\|_{L^\gamma(\Omega)} ^{\frac{p_j\eta_j p'}{p_j-q_j}} 
  \left(\mu_u(k)\right)^{
p'\left(1- \frac{p_j\eta_j}{(p_j-q_j)\gamma}-\frac{1}{p^\ast}\right)}
\end{aligned}
\end{equation}
for some large positive constant $C_\varepsilon$. 
 
From \eqref{fifa}, \eqref{ford0}, \eqref{car1}, \eqref{car2} and \eqref{euu}, choosing $\varepsilon \in (0,1/(2N))$ small and $k_0>1$ large, we get that
$$
\begin{aligned}
\sum_{\ell=1}^N\int_{A(k)}|\partial_\ell u|^{p_\ell}\, dx \le & 2 C_\varepsilon \left(\sum_{j \in I_1} \|u\|_{L^\gamma(\Omega)}^{p'(\theta_j-1)}\left(\mu_u(k)\right)^{p'\left(1-\frac{\theta_j-1}{\gamma}-\frac{1}{p^\ast}\right)}\right)\\
& + 2 C_\varepsilon \left(\sum_{j \in I_5} \|u\|_{L^\gamma(\Omega)}^{\frac{p_j \eta_j  p'}{p_j-q_j}}\left(\mu_u(k)\right)^{p'\left(1-\frac{p_j\eta_j}{(p_j-q_j)\gamma}-\frac{1}{p^\ast}\right)} \right)
\\
&+2\|u\|_{L^\gamma(\Omega)} \|f\|_{L^r(\Omega)} \left(\mu_u(k)\right)^{1-\frac{1}{\gamma}-\frac{1}{r}}
\\
\
\le& M \,\mu_u(k)^{\min\left\{p'\left(1-\frac{\theta-1}{\gamma}-\frac{1}{p^\ast}\right),\,
p'\left(1-\frac{\eta-1}{\gamma}-\frac{1}{p^\ast}\right),\,
1-\frac{1}{\gamma}-\frac{1}{r}\right\}},
\end{aligned}
$$
where 
$$\theta:=\max_{j \in I_1}\theta_j, \quad \eta:= \max_{j\in I_5}\left\{ \theta_j+\frac{q_j (\theta_j-m)}{p_j-q_j}\right\}=\max_{j\in I_5} \frac{p_j \eta_j}{p_j-q_j}+1>1,
$$
$C_\Omega$ is a positive constant depending on $|\Omega|$, $p$, $N$, and $r$, and
\begin{equation}\label{emme}
M:=C_\Omega \,\left(C_\varepsilon \left(\sum_{j \in I_1} \|u\|_{L^\gamma(\Omega)}^{p'(\theta_j-1)}\right)
+ C_\varepsilon \left(\sum_{j \in I_5} \|u\|_{L^\gamma(\Omega)}^{\frac{p_j \eta_j  p'}{p_j-q_j}}\right)
+\|u\|_{L^\gamma(\Omega)} \|f\|_{L^r(\Omega)} \right).
\end{equation}
Taking $h>k\ge k_0$ and using \eqref{art1} and \eqref{nnn}, we obtain that
\begin{equation} \label{neop}
\mu_u(h) \le \frac{\mathcal{S}^{p^\ast} M^{\frac{p^\ast}{p}}}{(h-k)^{p^\ast}}  \,\mu_u(k)^{\min \left\{ \widehat\beta(\gamma),\widetilde \beta(\gamma), \beta(\gamma) \right\}},
\end{equation} 
where we define $\widehat\beta(\gamma)$ and $ \beta(\gamma)$ as in \eqref{pap1}, and 
\begin{equation}\label{uau}
\widetilde \beta(\gamma):= \left(1-\frac{\eta-1}{\gamma}-\frac{1}{p^\ast}\right) \frac{p^\ast}{p-1} . 
\end{equation}

If $\theta\geq \eta$, then $\min\{\widetilde \beta(\gamma),\widehat \beta(\gamma)\}=\widehat \beta(\gamma)$. Then, as in the proof of Proposition~\ref{prop2}, we get $u\in L^\infty(\Omega)$. If $\theta<\eta$, then $\min\{\widetilde \beta(\gamma),\widehat \beta(\gamma)\}=\widetilde \beta(\gamma)$ 
and, reasoning as in the proof of Proposition~\ref{prop2} with $\eta$ instead of $\theta$, we conclude that $u\in L^\infty(\Omega)$. Indeed, in order to get a contradiction we need that $\eta<p$, which is true according to the definition of $I_5$. This finishes the proof of Proposition~\ref{prop3}. 
\end{proof}

\section{Proof of Theorem \ref{pre} concluded}

Assume that $u \in L^\gamma(\Omega)$ for some $\gamma \geq p^\ast$. In view of Proposition \ref{prop3}, Remarks \ref{rem1} and \ref{rem2}, in order to conclude the proof of Theorem \ref{pre}, we have to consider the remaining  possibility that $I_0\neq \emptyset$, where
$$
I_0:=\left\{1 \le j \le N: \mathfrak{a}_j=0,\,  q_j>0, \,  0<\theta_j< \max\{m,p\}\frac{(p_j-q_j)}{p_j}\right\}.
$$
Let $j \in I_0$ be arbitrary. Fix $\varepsilon>0$ small as needed and $k>k_0$ with $k_0>1$ sufficiently large.

($\bullet$) We first assume that $m \ge p$, which means that 
\begin{equation}\label{qq}
\frac{\theta_j p_j}{p_j-q_j}<m.
\end{equation}

 By H\"older's and Young's inequalities, there exists $C_\varepsilon>0$ such that
$$
\begin{aligned}
\int_{A(k)} |u|^{\theta_j-1}|\partial_j u|^{q_j} |G_k(u)|\, dx& \le \left(\int_{A(k)}|\partial_j u|^{p_j}\, dx\right)^{\frac{q_j}{p_j}}\left(\int_{A(k)}|u|^{\frac{(\theta_j-1) p_j}{p_j-q_j}}|G_k(u)|^{\frac{p_j}{p_j-q_j}}\, dx \right)^{\frac{p_j-q_j}{p_j}}
\\
& \le \varepsilon \int_{A(k)}|\partial_j u|^{p_j}\, dx+ C_\varepsilon \int_{A(k)}|u|^{\frac{\theta_j p_j}{p_j-q_j}-1}|G_k(u)|\, dx 
\\
& \le  \varepsilon \int_{A(k)}|\partial_j u|^{p_j}\, dx+ C_\varepsilon k^{\frac{\theta_j p_j}{p_j-q_j}-m}\int_{A(k)}|u|^{m-1}|G_k(u)|\, dx.
\end{aligned}
$$
Finally, since \eqref{qq} holds, the arguments in Proposition \ref{prop3} apply.

\medskip
($\bullet\bullet$) Let $m<p$. If $\theta_j<m(p_j-q_j)/p_j$, we conclude the proof with the same arguments as above. It remains to consider the case $j \in J_0$, where
$$
J_0:=\left\{1 \le j \le N: \mathfrak{a}_j=0,\,  q_j>0, \,\frac{m(p_j-q_j)}{p_j}\le \theta_j < \frac{p(p_j-q_j)}{p_j}\right\}.
$$
 By H\"older's and Young's inequalities, there exists $C_\varepsilon>0$ such that
$$
\begin{aligned}
\int_{A(k)} |u|^{\theta_j-1}|\partial_j u|^{q_j} |G_k(u)|\, dx& \le \left(\int_{A(k)}|\partial_j u|^{p_j}\, dx\right)^{\frac{q_j}{p_j}}\left(\int_{A(k)}|u|^{\frac{\theta_j p_j}{p_j-q_j}}\, dx \right)^{\frac{p_j-q_j}{p_j}}
\\
& \le \varepsilon \int_{A(k)}|\partial_j u|^{p_j}\, dx+ C_\varepsilon \int_{A(k)}|u|^{\frac{\theta_j p_j}{p_j-q_j}}\, dx 
\\
& \le  \varepsilon \int_{A(k)}|\partial_j u|^{p_j}\, dx+ C_\varepsilon k^{\frac{\theta_j p_j}{p_j-q_j}-p}\int_{A(k)}|u|^p\, dx
\\
& \le  \varepsilon \int_{A(k)}|\partial_j u|^{p_j}\, dx+ C_\varepsilon k^{\frac{\theta_j p_j}{p_j-q_j}-p} \left(\int_{A(k)}|u|^\gamma \, dx\right)^{\frac p \gamma}\left(\mu_u(k)\right)^{1-\frac p \gamma}.
\end{aligned}
$$
Let $\widehat\beta(\gamma)$ and $ \beta(\gamma)$ as in \eqref{pap1}, $\widetilde \beta(\gamma)$ as in \eqref{uau}, and $M$ as in \eqref{emme}, respectively. Taking $h>k\ge k_0$ and similar to the proof of \eqref{neop}, we obtain that
\begin{equation} \label{neop1}
\mu_u(h) \le \frac{\mathcal{S}^{p^\ast} \overline M^{\frac{p^\ast}{p}}}{(h-k)^{p^\ast}}  \,\mu_u(k)^{\min \left\{ \widehat\beta(\gamma),\widetilde \beta(\gamma), \beta(\gamma),\overline \beta(\gamma) \right\}},
\end{equation} 
where 
$$
\overline \beta(\gamma)=\left(\frac 1 p -\frac 1 \gamma \right)p^\ast,
$$
and
\begin{equation}\label{M}
\overline M = M+C_\Omega \left(C_\varepsilon k^{\max_{j \in J_0} \left\{\frac{\theta_jp_j}{p_j-q_j}\right\}-p}\|u\|_{L^\gamma(\Omega)}^p\right).
\end{equation}
It is easy to check that 
\begin{equation}\label{neop2}
\widehat \beta(\gamma)-\overline \beta(\gamma)>0, \quad \widetilde \beta(\gamma)-\overline \beta(\gamma)>0.
\end{equation}
Moreover, $\overline \beta(\gamma)-\beta(\gamma)$ is increasing in $\gamma$. Hence, we can apply an  iteration scheme analogous to the one  in Case $(E_2)$ of Proposition \ref{prop1}. Note that, for every $\gamma>p$, if $\overline \beta(\gamma)<1$, then $\gamma<p^\ast/(1-\overline \beta(\gamma))$. On the other hand, once the iteration procedure starts, it has to end in a finite number of steps. Indeed, if this were not the case, we would be able to construct two increasing sequences $\{\gamma_i\}_{i\ge 1}$ in $ (p,+\infty)$ and $\{\overline \beta(\gamma_i)\}_{i \ge 1}$ in $ (0,1)$, with $\overline \beta(\gamma_i) \to L \in (0,1]$ satisfying 
$$
L=\frac{p^\ast}{p}-1+L,
$$
which is a contradiction. 
As a consequence we get that 
\begin{equation*}\label{b1}
\mbox{there exists }\gamma=\gamma(N,p,r) \ge p^\ast \mbox{  such that }
\min\{\overline\beta(\gamma),\beta(\gamma)\}>1.
\end{equation*} 
Hence, from Lemma \ref{st} $i)$, we deduce that $u \in L^\infty(\Omega)$.
Moreover, using \eqref{M} and that $\max_{j \in J_0} \left\{\frac{\theta_jp_j}{p_j-q_j}\right\}-p<0$, we get \eqref{uinfty}.
This concludes the proof.


\medskip

\subsection*{Acknowledgments} The work on this project was finalised in June 2023 during the second author's visit to the Dipartimento di Matematica e Informatica, Universit\`a degli Studi di Palermo. The second author is very grateful for the support and hospitality of her co-author while carrying out research at her institution.

\medskip
\noindent {\bf Funding.} The first author has been supported by  the grant ``FFR 2023 Barbara Brandolini'', Universit\`a degli Studi di Palermo. The research of the second author has been supported by Australian Research Council under the Discovery Project Scheme (DP190102948 and DP220101816).



\medskip
\begin{thebibliography}{99}



\bib{ACCZ}{article}{
    AUTHOR = {Alberico, A.},
    author={I. Chlebicka},
    author={A. Cianchi}, 
    author={A. Zatorska-Goldstein},
    title={Fully anisotropic elliptic problems with minimally integrable data},
    JOURNAL = {Calc. Var. Partial Differential Equations},
    VOLUME = {58},
      YEAR = {2019},
    NUMBER = {6},
     PAGES = {Paper No. 186, 50},
     }


 \bib{AC}{article}{      
     AUTHOR = {Antontsev, S.},
     author={Chipot, M.},
     TITLE = {Anisotropic equations: uniqueness and existence results},
   JOURNAL = {Differential Integral Equations},
    VOLUME = {21},
      YEAR = {2008},
    NUMBER = {5-6},
     PAGES = {401--419},}

\bib{ADS}{book}{
    AUTHOR = {Antontsev, S. N.},
    author={D\'{\i}az, J. I. },
    author={Shmarev, S.},
     TITLE = {Energy methods for free boundary problems},
    SERIES = {Progress in Nonlinear Differential Equations and their
              Applications},
    VOLUME = {48},
      NOTE = {Applications to nonlinear PDEs and fluid mechanics},
 PUBLISHER = {Birkh\"{a}user Boston, Inc., Boston, MA},
      YEAR = {2002},}

\bib{AS}{book}{
	Author = { Antontsev, S.},
	author={Shmarev, S.},
	Editor = {M. Chipot and P. Quittner},
	Publisher = {North-Holland},
	Series = {Handbook of Differential Equations: Stationary Partial Differential Equations},
	Title = {Chapter 1 - Elliptic equations with anisotropic nonlinearity and nonstandard growth conditions},
	Volume = {3},
	Year = {2006},}
	

 \bib{BMP1}{article}{
    AUTHOR = {Boccardo, L.},
    author={ Murat, F.},
    author={ Puel, J.-P. },
     TITLE = {Existence de solutions faibles pour des \'{e}quations elliptiques
              quasi-lin\'{e}aires \`a croissance quadratique},
 BOOKTITLE = {Nonlinear partial differential equations and their
              applications. {C}oll\`ege de {F}rance {S}eminar, {V}ol. {IV}
              ({P}aris, 1981/1982)},
    SERIES = {Res. Notes in Math.},
    VOLUME = {84},
     PAGES = {19--73},
 PUBLISHER = {Pitman, Boston, Mass.-London},
      YEAR = {1983},
     
     
     
     }

		
    \bib{BMP}{article}{
    AUTHOR = {Boccardo, L.},
    author={Murat, F.},
    author={Puel, J.-P.},
     TITLE = {Existence of bounded solutions for nonlinear elliptic
              unilateral problems},
   JOURNAL = {Ann. Mat. Pura Appl. (4)},
    VOLUME = {152},
      YEAR = {1988},
     PAGES = {183--196},}
     
  
     
     
     \bib{BMP2}{article}{
     AUTHOR = {Boccardo, L.}
     author={ Murat, F. },
     author={Puel, J.-P.},
     TITLE = {{$L^\infty$} estimate for some nonlinear elliptic partial
              differential equations and application to an existence result},
   JOURNAL = {SIAM J. Math. Anal.},
  FJOURNAL = {SIAM Journal on Mathematical Analysis},
    VOLUME = {23},
      YEAR = {1992},
    NUMBER = {2},
     PAGES = {326--333},
     }

\bib{BST}{article}{
 AUTHOR = {Boccardo, L.},
 author={Segura de Le\'{o}n, S.},
 author={Trombetti,
              C.},
     TITLE = {Bounded and unbounded solutions for a class of quasi-linear
              elliptic problems with a quadratic gradient term},
   JOURNAL = {J. Math. Pures Appl. (9)},
  FJOURNAL = {Journal de Math\'{e}matiques Pures et Appliqu\'{e}es.
              Neuvi\`{e}me S\'{e}rie},
    VOLUME = {80},
      YEAR = {2001},
    NUMBER = {9},
     PAGES = {919--940},
}


\bib{baci}{article}{
author={Brandolini, B.},
AUTHOR = {C\^\i rstea, F. C.}, 
TITLE = {Anisotropic elliptic equations with gradient-dependent lower
              order terms and {$L^1$} data},
   JOURNAL = {Math. Eng.},
  FJOURNAL = {Mathematics in Engineering},
    VOLUME = {5},
      YEAR = {2023},
    NUMBER = {4},
     PAGES = {Paper No. 073, 33},

}

\bib{BC}{article}{
AUTHOR = {Brandolini, B.},
 AUTHOR = {C\^\i rstea, F. C.}, 
 title={Singular anisotropic elliptic equations with gradient-dependent lower order terms},
  journal={ NoDEA Nonlinear Differential Equations Appl.},
  volume={ 30}, 
  number={5},
    year={2023},
    pages={58 pp.}
  
}


    
    
\bib{CD1}{article}{
AUTHOR = {C\^{i}rstea, F. C.},
author={Du, Y.},
     TITLE = {General uniqueness results and variation speed for blow-up
              solutions of elliptic equations},
   JOURNAL = {Proc. London Math. Soc. (3)},
  FJOURNAL = {Proceedings of the London Mathematical Society. Third Series},
    VOLUME = {91},
      YEAR = {2005},
    NUMBER = {2},
     PAGES = {459--482},
     }    
    
    
\bib{CD}{article}{    
    AUTHOR = {C\^{i}rstea, F. C.},
    author={Du, Y.},
     TITLE = {Isolated singularities for weighted quasilinear elliptic
              equations},
   JOURNAL = {J. Funct. Anal.},
  FJOURNAL = {Journal of Functional Analysis},
    VOLUME = {259},
      YEAR = {2010},
    NUMBER = {1},
     PAGES = {174--202},}
    
  \bib{CV}{article}{ 
   AUTHOR = {C\^\i rstea, F. C.}, 
   author={V\'etois, J.},
     TITLE = {Fundamental solutions for anisotropic elliptic equations:
              existence and a priori estimates},
   JOURNAL = {Comm. Partial Differential Equations},
     VOLUME = {40},
      YEAR = {2015},
    NUMBER = {4},
     PAGES = {727--765},}

 \bib{dBFZ}{article}{
  author={   di Blasio, G.},
  author={Feo, F.},
  author={Zecca, G.},
  title={Regularity results for local solutions to some anisotropic elliptic equations},
  journal={ arXiv:2011.13412v1},
  year={2020},
  }
  
  \bib{DiCastro}{article}{
  AUTHOR = {Di Castro, A.},
     TITLE = {Anisotropic elliptic problems with natural growth terms},
   JOURNAL = {Manuscripta Math.},
  FJOURNAL = {Manuscripta Mathematica},
    VOLUME = {135},
      YEAR = {2011},
    NUMBER = {3-4},
     PAGES = {521--543},
  
  
  
  
  
  }
  
  \bib{Du}{book}{
  AUTHOR = {Du, Y.},
     TITLE = {Order structure and topological methods in nonlinear partial
              differential equations. {V}ol. 1},
    SERIES = {Series in Partial Differential Equations and Applications},
    VOLUME = {2},
      NOTE = {Maximum principles and applications},
 PUBLISHER = {World Scientific Publishing Co. Pte. Ltd., Hackensack, NJ},
      YEAR = {2006},
  
  
  }

  \bib{FVV}{article}{
   author={Feo, F.},
   author={Vazquez, J. L.},
   author={Volzone, B.},
   title={Anisotropic $p$-Laplacian Evolution of Fast Diffusion type},
   journal={Adv. Nonlinear Stud.},
   volume={21},
   year={2021}, 
   number={3},
   pages={523--555}
   
    }
   

   
         
 \bib{GM}{article}{
  AUTHOR = {Giachetti, D.},
  author={Murat, F.},
     TITLE = {An elliptic problem with a lower order term having singular
              behaviour},
   JOURNAL = {Boll. Unione Mat. Ital. (9)},
     VOLUME = {2},
      YEAR = {2009},
    NUMBER = {2},
     PAGES = {349--370},}
     
        \bib{GPS1}{article}{
     
     AUTHOR = {Giachetti, D.},
     author={Petitta, F.},
     author={de Le\'{o}n, Sergio
              Segura},
     TITLE = {Elliptic equations having a singular quadratic gradient term
              and a changing sign datum},
   JOURNAL = {Commun. Pure Appl. Anal.},
    VOLUME = {11},
      YEAR = {2012},
    NUMBER = {5},
     PAGES = {1875--1895},
     
     
     }
     
   
   \bib{GPS}{article}{
     AUTHOR = {Giachetti, D.},
     author={Petitta, F.},
     author={Segura de Le\'{o}n, S.},
     TITLE = {A priori estimates for elliptic problems with a strongly
              singular gradient term and a general datum},
   JOURNAL = {Differential Integral Equations},
     VOLUME = {26},
      YEAR = {2013},
    NUMBER = {9-10},
     PAGES = {913--948}, }
 

   
   \bib{GMP}{article}{   
   AUTHOR = {Grenon, N.}, 
   author={Murat, F.},
   author={Porretta, A.},
     TITLE = {A priori estimates and existence for elliptic equations with
              gradient dependent terms},
   JOURNAL = {Ann. Sc. Norm. Super. Pisa Cl. Sci. (5)},
     VOLUME = {13},
      YEAR = {2014},
    NUMBER = {1},
     PAGES = {137--205}, }  
     
 
 
\bib{LMS}{article}{
 AUTHOR = {Latorre, M.},
 author={Magliocca, M.},
 author={Segura de Le\'{o}n, S.},
     TITLE = {Regularizing effects concerning elliptic equations with a
              superlinear gradient term},
   JOURNAL = {Rev. Mat. Complut.},
  FJOURNAL = {Revista Matem\'{a}tica Complutense},
    VOLUME = {34},
      YEAR = {2021},
    NUMBER = {2},
     PAGES = {297--356},
}


    
    
    \bib{MT}{article}{
   author={Motreanu, D.},
   author={Tornatore, E.},
   title={ A Dirichlet problem with anisotropic principal part involving unbounded coefficients},
   Journal={Preprint},
   year={2023},
    
    
    
    } 
 
 \bib{PS}{article}{
 AUTHOR = {Porretta, A.},
 author={Segura de Le\'{o}n, S.},
     TITLE = {Nonlinear elliptic equations having a gradient term with
              natural growth},
   JOURNAL = {J. Math. Pures Appl. (9)},
  FJOURNAL = {Journal de Math\'{e}matiques Pures et Appliqu\'{e}es. Neuvi\`eme S\'{e}rie},
    VOLUME = {85},
      YEAR = {2006},
    NUMBER = {3},
     PAGES = {465--492},
 
 
 
 }
 

\bib{S}{article}{
AUTHOR = {Stampacchia, G.},
     TITLE = {Le probl\`eme de {D}irichlet pour les \'{e}quations elliptiques du
              second ordre \`a coefficients discontinus},
   JOURNAL = {Ann. Inst. Fourier (Grenoble)},
    VOLUME = {15},
      YEAR = {1965},
    NUMBER = {fasc. 1},
     PAGES = {189--258},}
        

           
\bib{T}{article}{  
    AUTHOR = {Troisi, M.},
     TITLE = {Teoremi di inclusione per spazi di Sobolev non isotropi},
   JOURNAL = {Ricerche Mat.},
    VOLUME = {18},
      YEAR = {1969},
     PAGES = {3--24},}
     
     
     \bib{V1}{book}{
     
      AUTHOR = {V\'{e}ron, L.},
     TITLE = {Singularities of solutions of second order quasilinear
              equations},
    SERIES = {Pitman Research Notes in Mathematics Series},
    VOLUME = {353},
 PUBLISHER = {Longman, Harlow},
      YEAR = {1996},
     PAGES = {viii+377},
     
     
     }
     
\bib{V2}{book}{
 AUTHOR = {V\'{e}ron, L.},
     TITLE = {Local and global aspects of quasilinear degenerate elliptic
              equations},
      NOTE = {Quasilinear elliptic singular problems},
 PUBLISHER = {World Scientific Publishing Co. Pte. Ltd., Hackensack, NJ},
      YEAR = {2017},
     PAGES = {xv+457},
     }     
     
\bib{V}{article}{
     AUTHOR = {V\'{e}tois, J.},
     TITLE = {Strong maximum principles for anisotropic elliptic and
              parabolic equations},
   JOURNAL = {Adv. Nonlinear Stud.},
    VOLUME = {12},
      YEAR = {2012},
    NUMBER = {1},
     PAGES = {101--114},}



\bib{WD}{article}{

 AUTHOR = {Wei, L.},
 author={Du, Y.},
     TITLE = {Exact singular behavior of positive solutions to nonlinear
              elliptic equations with a {H}ardy potential},
   JOURNAL = {J. Differential Equations},
  FJOURNAL = {Journal of Differential Equations},
    VOLUME = {262},
      YEAR = {2017},
    NUMBER = {7},
     PAGES = {3864--3886},
     }

\bib{ZD}{article}{
AUTHOR = {Zhang, X.},
author={Du, Y.},
     TITLE = {Sharp conditions for the existence of boundary blow-up
              solutions to the {M}onge-{A}mp\`{e}re equation},
   JOURNAL = {Calc. Var. Partial Differential Equations},
  FJOURNAL = {Calculus of Variations and Partial Differential Equations},
    VOLUME = {57},
      YEAR = {2018},
    NUMBER = {2},
     PAGES = {Paper No. 30, 24},


}

\end{thebibliography}
\end{document}